\documentclass{amsart}
\usepackage{amssymb}
\usepackage{amsfonts}
\usepackage{amssymb}
\usepackage{amsmath, accents}
\usepackage{amsthm}
\usepackage{enumerate}
\usepackage{tabularx}
\usepackage{centernot}
\usepackage{mathtools}
\usepackage{stmaryrd}
\usepackage{amsthm,amssymb}
\usepackage{etoolbox}
\usepackage{tikz}
\usepackage{amssymb}
\usetikzlibrary{matrix}
\usepackage{tikz-cd}
\usepackage{tikz}
\usepackage{marginnote}
\definecolor{mygray}{gray}{0.85}
\usepackage[backgroundcolor=mygray,colorinlistoftodos,prependcaption,textsize=small]{todonotes}
\usepackage{xcolor}

\renewcommand{\leq}{\leqslant}
\renewcommand{\geq}{\geqslant}

\makeatletter
\def\subsection{\@startsection{subsection}{3}%
  \z@{.5\linespacing\@plus.7\linespacing}{.3\linespacing}%
  {\bfseries\centering}}
\makeatother

\makeatletter
\def\subsubsection{\@startsection{subsubsection}{3}%
  \z@{.5\linespacing\@plus.7\linespacing}{.3\linespacing}%
  {\centering}}
\makeatother

\makeatletter
\def\myfnt{\ifx\protect\@typeset@protect\expandafter\footnote\else\expandafter\@gobble\fi}
\makeatother

\newtheorem{theorem}{Theorem}[section]

\newtheorem{corollary}[theorem]{Corollary}
\newtheorem{definition}[theorem]{Definition}
\newtheorem{lemma}[theorem]{Lemma}
\newtheorem{proposition}[theorem]{Proposition}

\newtheorem{fact}[theorem]{Fact}

\newtheorem{remark}[theorem]{Remark}

\newtheorem{convention}[theorem]{Convention}

\newcounter{claimcounter}
\numberwithin{claimcounter}{theorem}

\usepackage{eso-pic}
\usepackage{blindtext} 

\newcommand{\mrm}[1]{\mathrm{#1}}

\usepackage{pdfpages}
\usepackage{hyperref}
\newcounter{todocounter}

\begin{document}
\begin{abstract} 
We prove that the irreducible affine Coxeter groups are first-order rigid and deduce from this that they are profinitely rigid in the absolute sense. 
%
%
We then show that the first-order theory of any irreducible affine Coxeter group does not have a prime model. Finally, we prove that universal Coxeter groups of finite rank are homogeneous, and that the same applies to every hyperbolic (in the sense of Gromov) one-ended right-angled Coxeter group.
%
%
\end{abstract}

\title[Profinite rigidity of affine Coxeter groups]{Profinite rigidity of affine Coxeter groups}


\thanks{The authors were supported by a ``Grant for Internationalization'' from the University of Torino which made possible a visit of the second author to the first author in the Spring of 2023. The first author was also supported by project PRIN 2022 ``Models, sets and classifications", prot. 2022TECZJA}

\author{Gianluca Paolini}

\address{Department of Mathematics ``Giuseppe Peano'', University of Torino, Via Carlo Alberto 10, 10123, Italy.}
\email{gianluca.paolini@unito.it}

\author{Rizos Sklinos}

\address{Institute of Mathematics, Chinese Academy of Sciences, No.55 of Zhongguancun East Road, Beijing 100190, P. R. China.}

\email{rizos.sklinos@amss.ac.cn}

\date{\today}
\maketitle




\section{Introduction}
A central question in the model theory of groups is how much information the first-order theory of a group carries about the group and whether this information determines the group up to isomorphism. The main focus for us is on finitely generated groups (f.g. groups). Along this line of thought various notions of rigidity have been introduced. A f.g. $G$ is called {\em first-order rigid} if $G\cong H$ whenever $H$ is f.g. and elementarily equivalent to $G$, i.e., $G$ and $H$ have the same first-order theory. The term ``first-order rigidity" was introduced in \cite{ALM19}. The notion it defines coincides with the pre-existing notion of quasi-axiomatizability. In this paper we will use the term first-order rigid, since we believe that it describes the notion just introduced better. A f.g. group $G$ is called {\em quasi-finitely axiomatizable} (or QFA, for short) if there exists a first-order sentence $\sigma$ satisfied by $G$ such that $G\cong H$ whenever $H$ is f.g. and satisfies $\sigma$. A rigidity notion independent of first-order logic is the following: a f.g. residually finite group $G$ is called {\em profinitely rigid in the absolute sense} if $G\cong H$ whenever $H$ is f.g., residually finite and $\widehat{G}\cong\widehat{H}$, where, for any group $\Gamma$, $\widehat{\Gamma}$ denotes its profinite completion. Since in this paper we will not relativize profinite rigidity, we will just use the term without the ``in the absolute sense" qualification. The introduction of this notion is motivated by the general question of determination of what the family of finite images of a group can tell us about the group. In the case the group $G$ is assumed to be residually finite, the information on the family of isomorphism types of its finite images is captured by the profinite completion of $G$. Hence, it is straightforward that a f.g. residually finite group is profinitely rigid if it is determined (among the class of f.g. residually finite groups) by its finite quotients. Recently, profinite rigidity received considerable attention, as a major advancement was made in the field. In \cite{BMRS20} the first example of a full-sized profinitely rigid group was given (a group is full-sized if it contains a nonabelian free subgroup). 

\smallskip

It is not hard to show that any f.g. abelian group is first-order rigid, but only the finite ones are QFA. As a matter of fact, no infinite f.g. abelian-by-finite group is QFA \cite{Oger06}. Torsion-free f.g. nilpotent groups of class $2$ are first-order rigid, but not all of class $3$ are (see \cite{Hirshon77} and \cite{Oger91}). On the other hand, finitely generated nonabelian free groups are not first-order rigid, as a consequence of the positive answer to Tarski's question, i.e., any two nonabelian free groups have the same first-order theory \cite{Sela06, KM06}. The same is true of any nonabelian torsion-free hyperbolic group \cite{Sela09}. The interested reader can find more information on how the rigidity landscape looked like, up to 2007, in the excellent survey \cite{Nies07}. In recent years Avni-Lubotzky-Meiri  proved that irreducible non-uniform higher-rank characteristic zero arithmetic lattices (e.g. $\mrm{SL}_n(\mathbb Z)$ for $n\geq 3$) are first-order rigid \cite{ALM19}, this result rekindled the interest in first-order rigidity. In the aforementioned paper, the authors isolate some sufficient conditions ensuring first-order rigidity, and one of these properties is primality. A group $G$ is prime if it embeds elementarily in any group $H$ elementarily equivalent to it, i.e., $G$ embeds in $H$ in such a way that any first-order property true of a finite tuple from $G$ is preserved by the embedding. Primality of a f.g. group $G$ is actually equivalent to the following property: for any finite tuple from $G$ its orbit under $Aut(G)$ is $\emptyset$-definable.  The appearance of primality is not a coincidence, this notion has strong ties with rigidity. In particular, it was proved in \cite{OS06} that a f.g. nilpotent group is QFA if and only if it is prime. The existence of a QFA group which is not prime is a well-known open problem (see e.g. \cite{Nies07}). On the other hand, a counting argument shows that there are prime f.g. groups which are not QFA \cite{NS09}. If we relaxe the condition of $\emptyset$-definability of orbits of tuples to $\emptyset$-type definability, i.e., definability by possibly infinitely many $\emptyset$-formulas, then we get the notion of homogeneity. Examples of homogeneous groups are: free groups \cite{OH11, PS12}, rigid torsion-free hyperbolic groups \cite{PS12}, i.e., torsion-free hyperbolic groups that do not split as an amalgamated free product or as an HNN extension over a cyclic group, and free abelian groups (folklore). One might argue that in a sense:
$$\text{ first-order rigidity : homogeneity = QFAness : primality.}$$              

We now change the focus of our attention to profinite rigidity. Again, it is not hard to see that f.g. abelian groups have this property. On the other hand, there is a virtually cycic group that is not profinitely rigid \cite{Baumslag74}. In addition, the worst behaviour can already be exhibited by groups belonging to the class of metabelian groups \cite{Pickel74}. Finding a profinitely rigid group that is full-sized, has been proved a hard task, the first examples were discovered only recently in \cite{BMRS20}, as mentioned above. The driving force, along this line of thought, has been the following major open question, asked by Remeslennikov in \cite[Question 5.48]{MK10}: is a nonabelian free group profinitely rigid?

\smallskip

In this paper we are interested in rigidity, primality and homogeneity phenomena in Coxeter groups. This is a very well-studied class of groups with interesting properties that can be described combinatorially, algebraically and also geometrically. The study of Coxeter groups through the lens of model theory has been started in \cite{MPS22}. We restrict our attention to two almost disjoint families of Coxeter groups: affine Coxeter groups and hyperbolic right-angled Coxeter groups (at the intersection of these two families there is only the infinite dihedral group). We review the basic definitions concerning Coxeter groups in Section \ref{prel_sec} (see also \cite[Section~2]{MPS22}).

\smallskip
	
We first focus on affine Coxeter groups. We will restrict to the analysis of irreducible affine Coxeter groups. The geometric realization of an irreducible affine Coxeter group $W$ gives rise to a semidirect product decomposition of $W$ as $T \rtimes_\alpha W_0$, where $T$ is the so-called translation subgroup of $W$ and $W_0$ is an irreducible finite Coxeter group, so such groups are evidently abelian-by-finite. In particular, as mentioned above, they are not QFA. The main theorem of this paper is: 

\begin{theorem}\label{the_rigidity_theorem}
Any irreducible affine Coxeter groups is first-order rigid.
\end{theorem}

The work of Oger on elementary equivalence of abelian-by-finite groups \cite{Oger88} provides an interesting link to profinite rigidity. Indeed, Oger proved that two f.g. abelian-by-finite groups are elementarily equivalent if and only if they have isomorphic profinite completions. Using this logic machinery together with our previous theorem and some extra arguments, we obtain the following result:

\begin{theorem}\label{the_profinite_theorem}
Any irreducible affine Coxeter groups is profinitely rigid. 
\end{theorem}

To the best of our knowledge this is the first time that a first-order rigidity result has been used to obtain profinite rigidity. Furthermore, this partially answers Question 1.3 from \cite{varghese}. In that paper it is shown that irreducible affine Coxeter groups are pairwise profinitely distinguishable and it is asked if Coxeter groups are profinitely rigid. Even more interestingly, our Theorem~\ref{the_profinite_theorem} connects with the topic of mathematical crystallography. In fact, faithful split extensions of free abelian groups by finite groups are examples of crystallographic groups and the question of profinite rigidity of crystallographic groups has already been considered in the literature, most notably in \cite{crystallo_profinite}, where it is shown using a computer algebra program that any crystallographic group of dimension at most four is in fact profinitely rigid. Furthermore, as mentioned in \cite{crystallo_profinite}, it is known that there exists a faithful split extension of $\mathbb{Z}^{22}$ by the group $\mathbb{Z}_{23}$ which is {\em not} profinitely rigid. The problem of profinite rigidity of such groups is highly not trivial as it connects with the theory of genera from integral representation theory (cf. \cite[Chapter~15]{representation_theory_book}). All this makes our Theorem~\ref{the_profinite_theorem} even more relevant, as on one hand it is the first example of a profinite rigidity result for a class of crystallographic groups without bound on dimensions, and on the other it relies essentially on methods of logic and model theory.

\smallskip
\noindent

\smallskip

Focusing on the existence of prime models we prove the following negative result.



\begin{theorem}  
Let $W \cong \mathbb{Z}^n \rtimes_\alpha W_0$, with $W_0$ finite and  $\alpha:W_0\rightarrow GL_n(\mathbb{Z})$ injective. Then $\mrm{Th}(W)$ does not have a prime model.  
\end{theorem}

In particular, the first-order theory of any irreducible affine Coxeter group does not have a prime model. Note that Oger in \cite{Oger06} proved that infinite Abelian-by-finite groups are not prime. On the other hand, in \cite{MPS22} it was shown that any irreducible non-affine $2$-spherical Coxeter group of finite rank is prime.

%

\smallskip 
\noindent Passing to hyperbolic right-angled Coxeter groups we prove the following:

\begin{theorem} 
Universal f.g. Coxeter groups are homogeneous.
\end{theorem}

Universal Coxeter groups are virtually free and f.g. virtually free groups are known to be almost homogeneous (see \cite{And22} for this result and the notion of almost homogeneity). At this point it is worth mentioning that it is not known whether there exists a non-homogeneous f.g. virtually free group. In \cite{Andr22} it is proved that co-Hopfian f.g. virtually free groups are homogeneous.   
On the same line of thought, as one-ended hyperbolic groups are co-Hopfian (see \cite{Moi13}),  we show: 

\begin{theorem}
Hyperbolic one-ended right-angled Coxeter groups are homogeneous.  
\end{theorem}





%

A few words on the history of this paper. This paper was submitted to a journal on 19.12.2023 and it was posted on ArXiv on 01.07.2024, in concomitance with the posting of the paper \cite{paper_with_simon} (which uses some of the results of the present paper). After submission of the present paper on ArXiv we discovered that a paper was posted on ArXiv on 22.04.2024 (see \cite{other_paper}) announcing profinite rigidity of affine Coxeter groups.

\section{Preliminaries}\label{prel_sec}

\begin{definition}[Coxeter groups]\label{def_Coxeter_groups} Let $S$ be a set. A matrix $m: S \times S \rightarrow \{1, 2, . . . , \infty \}$ is called a {\em Coxeter matrix} if it satisfies:
	\begin{enumerate}[(1)]
	\item $m(s, s') = m(s' , s)$;
	\item $m(s, s') = 1 \Leftrightarrow s = s'$.
	\end{enumerate}
For such a matrix, let $S^2_{*} = \{(s, s') \in S^2 : m(s, s' ) < \infty \}$. A Coxeter matrix $m$ determines
a group $W$ with presentation:
$$
\begin{cases} \text{Generators:} \; \;  S \\
				\text{Relations:} \; \;   (ss')^{m(s,s')} = e, \text{ for all } (s, s' ) \in S^2_{*}.
\end{cases} $$
A group with a presentations as above is called a Coxeter group, and the pair $(W, S)$ is a called a Coxeter system. The rank of the Coxeter system $(W, S)$ is $|S|$.
\end{definition}

\begin{definition}\label{def_Coxeter_graph} In the context of Definition~\ref{def_Coxeter_groups}, the Coxeter matrix $m$ can equivalently be represented by a labeled graph $\Gamma$ whose node set is $S$ and whose set of edges $E_\Gamma$ is the set of pairs $\{s, s' \}$ such that $m(s, s') < \infty$, with label $m(s, s')$. Notice that some authors consider instead the graph $\Delta$ such that $s$ and $s'$ are adjacent iff $m(s, s ) \geq 3$. In order to try to avoid confusion we refer to the first graph as the Coxeter graph of $(W, S)$ (and usually denote it with the letter $\Gamma$), and to the second graph as the Coxeter diagram of $(W, S)$ (and usually denote it with the letter $\Delta$).
\end{definition}

	\begin{definition}[Right-angled Coxeter groups] Let $m$ be a Coxeter matrix and let $W$ be the corresponding Coxeter group. We say that $W$ is right-angled if the matrix $m$ has values in the set $\{ 1, 2, \infty\}$. In this case the Coxeter graph $\Gamma$ associated to $m$ is simply thought as a graph (instead of a labeled graph), with edges corresponding to the pairs $\{ s, s' \}$ such that $m(s, s') = 2$. By the {\em universal Coxeter group of rank $n$} we mean the Coxeter group whose associated Coxeter graph  is the graph with $n$ vertices and no edges, i.e., the free product of $n$ copies of the group $\mathbb{Z}/2\mathbb{Z}$.
	
\end{definition}

\begin{definition}\label{def_irreducible} Let $(W, S)$ be a Coxeter system with Coxeter diagram $\Delta$ (recall Definition~\ref{def_Coxeter_graph}). We say that $(W, S)$ is irreducible if $\Delta$ is connected.
\end{definition}

	\begin{remark} Let $(W, S)$ be a right-angled Coxeter system with Coxeter graph $\Gamma$ (recall Def.~\ref{def_Coxeter_graph}). Then $(W, S)$ is irreducible if and only if the complement of $\Gamma$ is connected.
\end{remark}

	\begin{fact} Let $(W, S)$  be a Coxeter system of finite rank. Then $W$ can be written uniquely as a product $W_1 \times \cdots \times W_n$ of irreducible special $S$-parabolic subgroups of $W$ (up to changing the order of the factors $W_i$, $i \in [1, n]$). In fact, if $S_1, ..., S_n$ are the connected components of the Coxeter diagram $\Delta$, then $W_i = \langle S_i \rangle_W$.
\end{fact}

	\begin{definition} Let $W$  be a Coxeter group. We say that $W$ is spherical if it is finite. We say that $W$ is affine if it is infinite and it has a representation as a discrete affine reflection group (see e.g. the classical reference \cite{humphreys} for details). 
\end{definition}

	\begin{fact} The irreducible affine Coxeter groups have been classified, a list of the corresponding Coxeter diagrams (recall Definition~\ref{def_Coxeter_graph}) can be found in Figure \ref{affine_figure}.	
	\begin{figure}
\begin{center}
\makebox[\textwidth][c]{\includegraphics[angle=270,origin=c, width=1.20\textwidth]{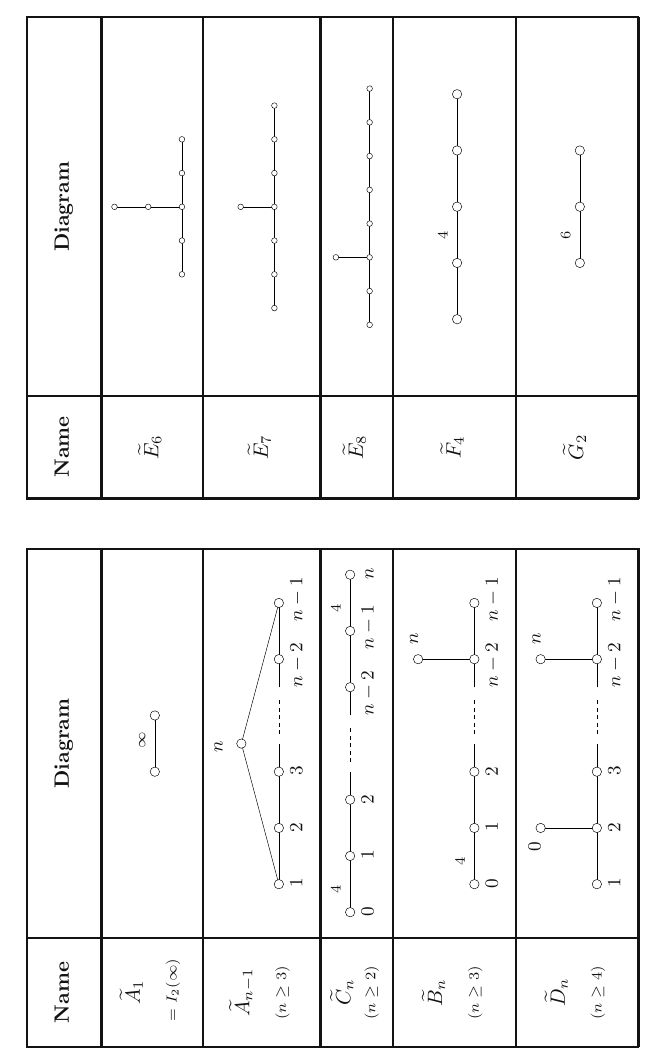}}\caption{The irreducible affine Coxeter groups}\label{affine_figure}
\end{center}
\end{figure}
\end{fact}

	\begin{fact}{\cite{moussong}} A right-angled Coxeter group is (Gromov) hyperbolic if and only if the corresponding graph does not contain an induced cycle of length four.
\end{fact}

\begin{convention}\label{semi_convention} 
When dealing with semidirect products we make the standard conventions. The notation $G = H \rtimes_\alpha W_0$ is used for the external semidirect product perspective. In particular, the action is given by the image of the homomorphism $\alpha: W_0 \rightarrow \mrm{Aut}(H)$. On the other hand, the notation $G = H \rtimes W_0$ is used for the internal semidirect product perspective. In particular, it is implicitly intended that $W_0$ acts by conjugation on $H$.
\end{convention}

\begin{fact}\label{fact_affine} Let $(W, S)$ be an irreducible affine Coxeter group.
Then there exists $N \trianglelefteq W$ and $W_0 \leq W$ such that:
\begin{enumerate}[(1)]
	\item $W = N \rtimes W_0$;
	\item $N \cong \mathbb{Z}^d$, for some $1 \leq d < \omega$;
	\item $W_0$ is a Weyl group (and so, in particular, a finite Coxeter group);
        \item the natural action of $W_0$ on $N$ is faithful.
\end{enumerate}
\end{fact}

    \begin{proof} Items (1)-(3) are \cite[Proposition~2, pg. 146]{brown}. Concerning (4), the action of $W_0$ on $\mathbb{Z}^n$ gives us an action of $W_0$ on $\mathbb{R}^n := \mathbb{Z}^n \otimes \mathbb{R}$
and the action of $W_0$ on $\mathbb{R}^n$ is precisely the geometric representation
of the finite Coxeter group $W_0$ (see \cite[Proposition~4.2]{humphreys}). Thus, to conclude it suffices to remark that this action is faithful (cf. \cite[Theorem~1.12]{humphreys}).
\end{proof}

\section{Affine Coxeter groups}\label{Affine}

\subsection{Prime models}


\begin{fact}\label{fact_matrix} Let $G$ be an Abelian group and let $M$ be an $n \times n$ matrix with integer coefficients. Then we can associate to $M$ an endomorphism, $f_M:G^n\rightarrow G^n$, as follows:
	$$(g_1, ..., g_n) \mapsto M \cdot \begin{pmatrix}
           g_{1} \\
           g_{2} \\
           \vdots \\
           g_{n}
         \end{pmatrix}$$
Furthermore, if $M \in \mrm{GL}_n(\mathbb{Z})$, then $f_M$ is an automorphism of $G^n$.
\end{fact}


In the light of the above fact we observe the following. Given a morphism $\alpha:W\rightarrow GL_n(\mathbb{Z})$ and an abelian group $G$, we associate to the couple $(\alpha, G)$ a morphism $\alpha_G:W\rightarrow Aut(G^n)$, where $\alpha_G$ maps an element $w\in W$ to the automorphism associated by Fact \ref{fact_matrix} to the matrix $\alpha(w)$. In particular, for $\alpha:W\rightarrow GL_n(\mathbb{Z})$ and $G$ an Abelian group, the notation $G^n \rtimes_{\alpha} W$ meaning $G^n \rtimes_{\alpha_G} W$,  makes sense.


Following \cite[Chapter~3]{hodges}, we denote by $\mrm{EF}_k[A, B]$ the Ehrenfeucht-Fra\"iss\'e game of length $k$ with winning condition the preservation of unnested atomic formulas.
	
\begin{fact}{\cite[Corollary 3.3.3, p.105]{hodges}} Let $A, B$ be $L$-structures (of finite language $L$). Then TFAE:
	\begin{enumerate}[(1)]
	\item $A \equiv B$;
	\item for all $k < \omega$, Player II has a winning strategy in the game $\mrm{EF}_k[A, B]$.
	\end{enumerate}
\end{fact}

\begin{theorem}\label{the_game_theorem} Suppose that $G \equiv \mathbb{Z}$, $W_0$ is a finite group and $\alpha: W_0 \rightarrow GL_n(\mathbb{Z})$ is an homomorphism. Then $\mathbb{Z}^n \rtimes_{\alpha} W_0 \equiv G^n \rtimes_{\alpha} W_0$.
\end{theorem}

\begin{proof} Since $G \equiv \mathbb{Z}$ we know that, for all $k < \omega$, Player II has a winning strategy in the game $\mrm{EF}_k[\mathbb{Z}, G]$. Using this we want to show that, for all $m < \omega$, Player II has a winning strategy in the game $\mrm{EF}_m[\mathbb{Z}^n \rtimes_{\alpha} W_0, G^n \rtimes_{\alpha} W_0]$. To this extent, we fix $m < \omega$ and show how to play the game $\mrm{EF}_m[\mathbb{Z}^n \rtimes_{\alpha} W_0, G^n \rtimes_{\alpha} W_0]$. 

\smallskip 
\noindent The general strategy is as follows: when playing a play of the game $\mrm{EF}_m[\mathbb{Z}^n \rtimes_{\alpha} W, G^n \rtimes_{\alpha} W_0]$, on the side we play an auxiliary play of the game $\mrm{EF}_{f(m)}[\mathbb{Z}, G]$ (for some $f(m) \geq m$) which will ensure that the play in $\mrm{EF}_m[\mathbb{Z}^n \rtimes_{\alpha} W_0, G^n \rtimes_{\alpha} W_0]$ is winning for Player~II. 
	
\smallskip
\noindent Let $W_0=\{w_1, ..., w_t\}$. Now, given $\bar{b}=(b_1, b_2, \ldots, b_n)$ of length $n$, which is either in $\mathbb{Z}^n$ or in $G^n$, we will extend $\bar{b}$ to a tuple $\bar{b}^+$ in some cartesian power of $\mathbb{Z}$ or $G$ respectively, as follows: we essentially concatenate all coordinates of images of $\bar b$ by the automorphisms that correspond to $\alpha(w_i)$, for each $i\leq t$. More technically, if $M_r=(a^r_{ij})$ is the matrix associated to $\alpha(w_r)$, for $r\leq t$, then $\bar{b}^+$ is the tuple:
\begin{enumerate}[(1)]
	\item $\bar{b}$; concatenated with  
	\item $b_1 + b_1$, $b_1 + b_1 + b_1$, \ldots, $a^1_{11}b_1$; concatenated with 
	\item $b_2 + b_2$, $b_2 + b_2 + b_2$, \ldots, $a^1_{12}b_2$; concatenated with \\
	$\vdots$
	\item $b_n + b_n$, $b_n + b_n + b_n$, \ldots, $a^1_{1n}b_n$; concatenated with 
	\item $a^1_{11}b_1 + a^1_{12}b_2$, \ldots, $\sum^n_{j = 1} a^1_{1j} b_j$; concatenated with 
	\item we do the same for each row of the matrix $M_1$; concatenated with 
    \item we do the same for each row of the matrix $M_2$; concatenated with \\ 
	$\vdots$
	\item We do the same for each row of the matrix $M_r$.
\end{enumerate}
We note in passing that one of the matrices is the identity matrix. 
We also observe that the length, say $R$, of the tuple $\bar b^+$ only depends on the integer values of the coefficients of the matrices $M_r$, for $r\leq t$. With this at hand, we finally explain the winning strategy for Player II in the game $\mrm{EF}_m[\mathbb{Z}^n \rtimes_{\alpha} W, G^n \rtimes_{\alpha} W_0]$.

\smallskip
\noindent Suppose Player I plays $(\bar{b}, w_r)$ in the game $\mrm{EF}_{m}[\mathbb{Z}^n \rtimes_{\alpha} W_0, G^n \rtimes_{\alpha} W_0]$. Player II will guide their choices by playing another game on the side. If $\bar b$ belongs to $\mathbb{Z}^n$, then Player II will play $\bar b^+$ (a tuple in $\mathbb{Z}^R$), element by element, in the game $\mrm{EF}_{f(m)}[\mathbb{Z}, G]$, where $f(m)=R\cdot m$ and will reply back in the original game the answers he gets, say $(b_1', \ldots, b_n')$, to the first $n$ elements he played in the auxiliary game, plus $w_r$. Likewise if Player I had chosen a tuple in $G^n$. 

\smallskip
\noindent We argue that this strategy is winning for Player II. To this extent, let $((\bar{b}_1, w_{i_1}), $ $ (\bar{b}_2, w_{i_2}), \ldots, (\bar{b}_m, w_{i_m}); (\bar{b}_1', w_{i_1}), (\bar{b}_2', w_{i_2}), \ldots, (\bar{b}'_m, w_{i_m}))$ be a play at the end of the game, where $\bar b_i$ belongs to $\mathbb Z^n$ and $\bar b_i'$ belongs to $G^n$, for all $i\leq m$. To avoid any confusion in the forementioned play we have collected the tuples of each structure in the same side of ; independent of which Player played them. In addtion, the auxiliary game will look like $(b_{11}, b_{12}, \ldots, b_{1n},\ldots, b_{1R}, b_{21}, b_{22}, \ldots, b_{2n},\ldots, b_{2R}, \ldots, b_{m1}, b_{m2},$ $ \ldots, b_{mn}, \ldots b_{mR}\ ; \ b_{11}', b_{12}', \ldots, b_{1n}',\ldots, b_{1R}', b_{21}', b_{22}', \ldots, b_{2n}',\ldots, b_{2R}', \ldots, b_{m1}', b_{m2}', \ldots,$ $ b_{mn}', \ldots, b_{mR}')$. Blocks of $R$-many elements starting from the beginning have been played by the same Player, but each side of ; might contain plays from a different player as in the main game. We need to show that all the unnested atomic formulas are preserved. Recall that in the language of groups the unnested atomic formulas are formulas of the form $x=y$, $1=y$, $x^{-1}=y$ and  $x_0\cdot x_1=y$. 
We check each of these cases. 
\begin{itemize}

\item[{\underline{Case 1}}:] $x=y$. We need to show $(\bar b_j, w_{i_j})=(\bar b_k, w_{i_k})$ if and only if $(\bar b_j', w_{i_j})=(\bar b_k', w_{i_k})$, for any $j,k\leq m$. We observe that $(\bar b_j, w_{i_j})=(\bar b_k, w_{i_k})$ if and only if $(b_{j1}, b_{j2}, \ldots, b_{jn})=(b_{k1}, b_{k2}, \ldots, b_{kn})$ and $w_{ij}=w_{ik}$ and since all the $b's$ in the above tuples are  part of the auxiliary game (in which unnested formulas are preserved) we have that $(b_{j1}, b_{j2}, \ldots, b_{jn})=(b_{k1}, b_{k2}, \ldots, b_{kn})$ if and only if $(b_{j1}', b_{j2}', \ldots, b_{jn}')=(b_{k1}', b_{k2}', \ldots,$ $ b_{kn}')$, and thus we can conclude.

\item[{\underline{Case 2}}:] $1=y$. Case 2 is similar to Case 1, by just observing that $(\bar b_j, w_{i_j})=1$ if and only if $b_{j1}=0$ and $b_{j2}=0$ and \ldots and $b_{jn}=0$ and $w_{i_j}=e$, where $e$ is the identity element of $W_0$. Hence $(\bar b_j, w_{ij})=1$ if and only if $(\bar b_j', w_{ij})=1$.
\\ \\
\item[{\underline{Case 3}}:] $x^{-1}=y$. In this case we need to show $(\bar b_j, w_{i_j})^{-1}=(\bar b_k, w_{i_k})$ if and only if $(\bar b_j', w_{i_j})^{-1}=(\bar b_k', w_{i_k})$, for any $j,k\leq m$. We note that $(\bar b,w)^{-1}=(\alpha(w^{-1})(-\bar b), w^{-1})$. Hence $(\bar b_j, w_{i_j})^{-1}=(\bar b_k, w_{i_k})$ if and only if $-\alpha(w_{i_j}^{-1})(\bar b_j)$ $=\bar b_k$ and $w_{i_j}^{-1}=w_{i_k}$. Now we observe that each coordinate of $\alpha(w_{i_j}^{-1})(\bar b_j)=(a^t_{11}b_{j1}+a^t_{12}b_{j2}+\ldots+ a^t_{1n}b_{jn}, \ldots, a^t_{n1}b_{j1}+a^t_{n2}b_{j2}+\ldots+ a^t_{nn}b_{jn}$ (we assume $w_{i_j}^{-1}=w_t$ for some $t\leq r$) is part of the auxiliary game and in particular it corresponds to the same coordinate of $\alpha(w_{i_j}^{-1})(\bar b_j')$. The latter holds because we have chosen to add to the auxiliary game all the sums of the form $b_{j1}$, $b_{j1}+b_{j1}, \ldots, a^t_{1n}b_{j1}$ etc. Therefore, since formulas of the form $-x=y$ (additive notation) are preserved in the auxiliary game we get that $a^t_{11}b_{j1}+a^t_{12}b_{j2}+\ldots+ a^t_{1n}b_{jn}=-b_{k1}$ if and only if $a^t_{11}b_{j1}'+a^t_{12}b_{j2}'+\ldots+ a^t_{1n}b_{jn}'=-b_{k1}'$ and $\ldots$ and  $a^t_{n1}b_{j1}+a^t_{n2}b_{j2}+\ldots+ a^t_{nn}b_{jn}=-b_{kn}$ if and only if $a^t_{n1}b_{j1}'+a^t_{n2}b_{j2}'+\ldots+ a^t_{nn}b_{jn}'=-b_{kn}'$. Since the equality $w_{i_j}^{-1}=w_{i_k}$ stays the same in both structures, we may conclude.

\item[{\underline{Case 4}}:] $x_0\cdot x_1=y$. This is the most interesting case, but not so much different from the previous one. We need to prove that $(\bar b_j, w_{i_j})\cdot (\bar b_k, w_{i_k})= (\bar b_r, w_{i_r})$ if and only if  
$(\bar b_j', w_{i_j})\cdot (\bar b_k', w_{i_k})= (\bar b_r', w_{i_r})$, for any $j,k,r\leq m$. We note that 
$(\bar b_1, w_1)\cdot (\bar b_2, w_2)=(\bar b_1 + \alpha(w_1)(\bar b_2), w_1w_2)$. Hence, 
$(\bar b_j, w_{i_j})\cdot (\bar b_k, w_{i_k})= (\bar b_r, w_{i_r})$ if and only if $(\bar b_j + \alpha(w_{i_j})(\bar b_k), w_{i_j}w_{i_k})=(\bar b_r, w_{i_r})$. Now as in the previous case, each coordinate of $\alpha(w_{i_j})(\bar b_k)$ is part of the auxiliary game and it corresponds to the same coordinate in $\alpha(w_{i_j})(\bar b_k')$. In particular, if $w_{i_j}=w_t$ for some $t\leq r$, we get $b_{j1}+a^t_{11}b_{k1}+a^t_{12}b_{k2}+\ldots+ a^t_{1n}b_{kn}=b_{r1}$ if and only if 
$b_{j1}'+a^t_{11}b_{k1}'+a^t_{12}b_{k2}'+\ldots+ a^t_{1n}b_{kn}'=b_{r1}'$ and \ldots and 

$b_{jn}+a^t_{n1}b_{k1}+a^t_{n2}b_{k2}+\ldots+ a^t_{nn}b_{kn}=b_{rn}$ if and only if 
$b_{jn}'+a^t_{n1}b_{k1}'+a^t_{n2}b_{k2}'+\ldots+ a^t_{nn}b_{kn}'=b_{rn}'$. Since the equality $w_{i_j}w_{i_k}=w_{i_r}$ stays the same in both structures, we may conclude.

\end{itemize}

\end{proof}

\begin{proposition}\label{the_prop_def_transla} Let $G = T \rtimes_\alpha W_0$, with $T$ torsion-free abelian,  $W_0$ finite and $\alpha:W_0\rightarrow Aut(T)$ injective. Then $T$ is $\emptyset$-definable in $G$. As a matter of fact, if $W_0$ has order $m$, then  $T$ is definable by the following first-order formula:
    $$ \phi(x):= \forall y ([x,y^m]=1)$$
\end{proposition}

\begin{proof} 
We observe that since $T \ (=T\times \{e\})$ is a normal subgroup of $G$ of index $m$, we have that for every $g\in G$, $g^m$ belongs to $T$. In addition, $T$ is abelian, thus if $a\in T$, then $a$ commutes with every $m$-th power of $G$ and consequently $G\models \phi(a)$. Hence, $T\subseteq \phi(G)$.

\smallskip
\noindent Suppose now that there exists $g = (a, w) \in G$ with $w \neq e$ such that $G\models \phi(g)$. In particular, for every $(b,e) \in T$, we have that the following holds:
$$[(a, w), (b, e)^m] = [(a, w), (mb, e)] = (0, e).$$
But then, for every $(b,e) \in T$, we have that the following holds:
	$$(a, w)(mb, e) = (a + \alpha(w)(mb), w),$$
 and 
 $$(mb, e)(a, w)=(mb + a, w).$$

\smallskip
\noindent Hence, $(mb + a, w)=(a + \alpha(w)(mb), w)$. Thus, for every $(b,e) \in T$ we have that $\alpha(w)(mb) = mb$, but $\alpha(w) \in \mrm{Aut}(T)$ and so we have that $\alpha(w)(mb)=m\alpha(w)(b)$, and consequently $m\alpha(w)(b) = mb$. Now, since $T$ is torsion-free, it follows that $\alpha(w)(b) = b$ for every $b$ (such that $(b,e)\in T$). But this is a contradiction, as $\alpha$ was assumed to be injective and $w\neq e$.
\end{proof}

To prove the next result we will follow and adapt the main idea of \cite{BBGK73}. There it was shown that the first-order theory of the infinite cyclic group does not admit a prime model. We denote by $\mathbb{Z}_{(p)}$, where $p$ is a prime, the (additive) group of all rationals whose denominator (in lowest terms) is not divisible by $p$. We will use the fact, obtained by Szmeliew's characterization of complete theories of abelian groups, that $\bigoplus_{p \in \mathbb{P}} \mathbb{Z}_{(p)}$, where $\mathbb{P}$ is the set of primes, is elementarily equivalent to~$\mathbb{Z}$. 

\begin{theorem} \label{Nonprimality}
Let $W = \mathbb{Z}^n \rtimes_{\alpha} W_0$, where $W_0$ is finite and $\alpha:W_0\rightarrow GL_n(\mathbb{Z})$ injective. Then $\mrm{Th}(W)$ does not have a prime model.
\end{theorem}

\begin{proof} By Proposition \ref{the_prop_def_transla} there exists a first-order formula $\varphi(x)$ such that $\varphi(W)$ defines a subgroup of $W$ isomorphic to $\mathbb{Z}^n$. Therefore, as a structure, $(\varphi(W),\cdot)$, it satisfies the axioms of $\mrm{Th}(\mathbb{Z}^n)$. Since $\varphi(x)$ is defined over the $\emptyset$, we get that for every model $G$ of $\mrm{Th}(W)$, $(\varphi(G),\cdot) \models \mrm{Th}(\mathbb{Z}^n)$. For the rest of the proof we will treat the solution set of $\phi(x)$ in any model of $\mrm{Th}(W)$ as a model of $\mrm{Th}(\mathbb{Z}^n)$.

\smallskip
\noindent We assume for a contradiction that $\mrm{Th}(W)$ admits a prime model, say $\mathcal{A}$. Then $\mathcal{A}$ elementarily embeds into $W$ and consequently $\varphi(\mathcal{A})$ elementarily embeds into $\mathbb{Z}^n$. The latter implies that $\varphi(\mathcal{A})$ is isomorphic to $\mathbb{Z}^n$, as $\mathbb{Z}^n$ is strictly minimal, i.e. it does not have a proper elementary subgroup. By Theorem \ref{the_game_theorem}, since $\mathbb{Z}\equiv \bigoplus_{p \in \mathbb{P}} \mathbb{Z}_{(p)}$, we get that $H:=(\bigoplus_{p \in \mathbb{P}} \mathbb{Z}_{(p)})^n \rtimes_\alpha W_0$ is a model of $\mrm{Th}(W)$. By Proposition \ref{the_prop_def_transla} $\varphi(H)$ is isomorphic to $(\bigoplus_{p \in \mathbb{P}} \mathbb{Z}_{(p)})^n$ and consequently $\varphi(\mathcal{A})$ elementarily embeds into $(\bigoplus_{p \in \mathbb{P}} \mathbb{Z}_{(p)})^n$.

\smallskip
\noindent Finally, the element $(1, ..., 1) \in \mathbb{Z}^n$ is not divisible by any prime, but each non-zero element of $(\bigoplus_{p \in \mathbb{P}} \mathbb{Z}_{(p)})^n$ is divisible by infinitely many primes. In particular, $\mathcal{A}$ does not embed in $(\bigoplus_{p \in \mathbb{P}} \mathbb{Z}_{(p)})^n \rtimes_\alpha W_0$ elementarily.  
\end{proof}

\subsection{Rigidity}

The main results of this section are that irreducible affine Coxeter groups are first-order and profinitely rigid. 

\begin{proposition}\label{rigidity_prop} Let $W = \mathbb{Z}^n \rtimes_{\alpha} W_0$, where $W_0$ is finite and $\alpha:W_0\rightarrow GL_n(\mathbb{Z})$ injective.
Let $H$ be a finitely generated group elementarily equivalent to $W$. Then $H \cong \mathbb{Z}^n \rtimes_{\beta} W_0$ where $\beta:W_0\rightarrow \mathbb{Z}^n$ is injective. 
%
\end{proposition}

\begin{proof} Let $W_0 = \{w_1, ..., w_k\}$ (with $w_1=e$) and, for every $1\leq i\leq k$, let $M_i = (a^i_{j \ell})$ be the matrix in $GL_n(\mathbb{Z})$ which is the image of $w_i$.
Notice that as $W_0$ is assumed to be finite and since $\mathbb{Z}^n$ is torsion-free  we have that there is a bound on the torsion elements of $W$. Indeed, if the order of $w\in W_0$ is $k$, and $(\bar b,w)^n=(0,e)$, then $k$ divides $n$, say $k\cdot\lambda=n$, and $(\bar b,w)^n=(\lambda(\bar b+\alpha(w)(\bar b)+\ldots +\alpha(w^{k-1})(\bar b)),w^n)$, therefore, since $\lambda(\bar b+\alpha(w)(\bar b)+\ldots+ \alpha(w^{k-1})(\bar b))=0$, we get, by torsion-freeness, $\bar b+\alpha(w)(\bar b)+\ldots+ \alpha(w^{k-1})(\bar b)=0$. Therefore, the order of $(\bar b, w)$ is $k$. Now, let $m < \omega$ be the exponent of $W_0$, i.e. the smallest positive integer such that $w^m=e$ for every  $w \in W_0$. Then, by the previous argument, this $m$ serves as a bound on the orders of torsion elements in $W$. Let $\psi_W(y_1, ..., y_n, z_1, ..., z_k)$ be the first-order formula (over the $\emptyset$) expressing the following:
\ \\
	\begin{enumerate}[(1)]
	\item $\{z_1, ..., z_k\}$ is a subgroup isomorphic to $W_0$ with $z_1=e$;
	\item the centralizer of $\{y_1, ..., y_n\}$, $Z(y_1, ..., y_n)$, is an abelian subgroup;
	\item the whole group is the semidirect product $Z(y_1, ..., y_n) \rtimes \{z_1, ..., z_k\}$, i.e. the group $Z(y_1, ..., y_n)$ is a normal subgroup and for every element $x$ there are unique $y\in Z(y_1, ..., y_n)$ and unique $z\in \{z_1, ..., z_k\}$ such that $x=yz$;
	\item for every $1\leq \ell \leq m$, if $x \in Z(y_1, ..., y_n)\setminus\{e\}$, then $x^\ell \neq e$;
	\item $|Z(y_1, ..., y_n)/2Z(y_1, ..., y_n)| = 2^n$;
	\item for every $2\leq i \leq k$ there is $h \in Z(y_1, ..., y_n)$ such that $z_ihz_i^{-1} \neq h$. 
	\end{enumerate}
\ \\ 
Recall that $m$ was fixed at the beginning of the proof. For $n < \omega$, let $\varphi_n$ be the first-order formula saying ``if $x^n = e$, then $x^\ell = e$ for some $\ell \leq m$". 

\smallskip
\noindent{\bf Claim:} If we let:
$$\chi = \exists y_1, ..., y_n \exists z_1, ..., z_k \bigl(\psi_W(y_1, ..., y_n, z_1, ..., z_k)\bigr),$$
then we have that:
$$W \models \chi \text{ and } W \models \varphi_n, \text{ for every $m < n < \omega$}.$$
\noindent{\bf Proof (of Claim):}
We only need to show that if we let $\{y_1, ..., y_n\}$ be a free basis of $\mathbb{Z}^n$, then the centralizer $Z(y_1, ..., y_2)$ in $W = \mathbb{Z}^n \rtimes_{\alpha} W_0$, is the subgroup $ \mathbb{Z}^n\times\{e\}$. Indeed, let $(a, w) \in \mathbb{Z}^n \rtimes_{\alpha} W_0$, $b \in \mathbb{Z}^n$ and suppose that:
$$(a, w)(b, e) = (a + \alpha(w)(b), w) = (b + a, w) = (b, e)(a, w),$$
Then $a + \alpha(w)(b) = b + a = a + b$ and so $\alpha(w)(b) = b$. Hence, if $(a, w) \in \mathbb{Z}^n \rtimes_{\alpha} W_0$ commutes with $(y_i, e)$, for every $1\leq i \leq n$,  then $\alpha(w)(y_i) = y_i$, for every $1\leq i \leq n$, but as $\{y_1, ..., y_n\}$ is a basis of $\mathbb{Z}^n$ necessarily $\alpha(w) = \mrm{id}_{\mathbb{Z}^n}$. Since $\alpha$ is injective we get that $w=e$, and we can conclude. \qed

\smallskip \noindent
Suppose now that $H$ is finitely generated and elementarily equivalent to $W$. Then:
$$H \models \chi \text{ and } H \models \varphi_n, \text{ for every $m < n < \omega$}.$$
Using this it is easy to see then that $\mathbb{Z}^n \rtimes_{\beta} W_0$ for some $\beta: W_0 \rightarrow \mrm{Aut}(\mathbb{Z}^n)$. Indeed, $H$, from property (3) is the semidirect product of an Abelian (normal) subgroup $A$ (property (2)) with a finite subgroup $F$ isomorphic to $W_0$ (property (1)). Since $H$ is finitely generated and $A$ is normal of finite index in $H$, we get that $A$ is finitely generated. Now from property (4) and $\varphi_n$ we get that $A$ is torsion-free and finally from property (5) we get that it is isomorphic to $\mathbb{Z}^n$. Hence $H$ is isomorphic to a semidirect product of the form $\mathbb{Z}^n \rtimes_{\beta} W_0$. Lastly, property (6) implies that  $\beta:W_0\rightarrow \mathbb{Z}^n$ is injective. 

\smallskip

\end{proof}


We will use the above proposition and known results in order to show that profinite rigidity is equivalent to first-order rigidity for any irreducible affine Coxeter group.


%

Oger in \cite{Oger88} characterized elementary equivalence of abelian-by-finite f.g. groups through their profinite completions. 

\begin{fact}[Oger]\label{Oger}
Let $G,H$ be f.g. Abelian-by-finite groups. Then $G\equiv H$ if and only if $\widehat{G}\cong \widehat{H}$.    
\end{fact}

We will also be needing the following fact from \cite{grune}.

\begin{fact}(\cite[Proposition 2.6]{grune})\label{Grune}
Let $N, H$ be finitely generated residually finite groups and $\alpha:H\rightarrow Aut(N)$ be a homomorphism. Then 
$$ \widehat{N\rtimes_{\alpha} H} \cong \widehat{N} \rtimes_{\hat{\alpha}} \widehat{H} $$

\end{fact}

%

Using the above facts together with Proposition \ref{rigidity_prop} we get.

\begin{theorem}\label{conclusion_equiv_semi} Let $G \cong \mathbb{Z}^n \rtimes_\alpha W_0$, where  $\alpha:W_0\rightarrow GL_n(\mathbb{Z})$ injective. Let $H$ be a finitely generated residually finite group. TFAE:
	\begin{enumerate}[(1)]
	\item $H$ is elementarily equivalent to $G$;
	\item $H$ and $G$ have isomorphic profinite completions;
	\end{enumerate}
\end{theorem}

\begin{proof}
$(1) \Rightarrow (2)$.  Since $H$ is elementarilly equivalent to $G$ we get, by Proposition \ref{rigidity_prop}, that $H$ is Abelian-by-finite as well. Thus, by Fact \ref{Oger}, $H$ and $G$ have isomorphic profinite completions. We note that in this direction we do not use the assumption that $H$ is residually finite.

\smallskip \noindent $(2) \Rightarrow (1)$. Since $H$ is residually finite it embeds in its profinite completion. But its profinite completion, by Fact \ref{Grune},  is isomorphic to $\widehat{\mathbb{Z}^n}\rtimes_{\hat{\alpha}}W_0$ (note that $\widehat{W_0}=W_0$, for $W_0$ is finite). In particular, $H$ is Abelian-by-finite as well. By Fact \ref{Oger} $H$ is elementarily equivalent to $G$.  
\end{proof}

It follows immediately that:

\begin{corollary}\label{RigProf}
Let $G$ be an irreducible affine Coxeter group. Then $G$ is first-order rigid if and only if it is profinitely rigid.
\end{corollary}

\begin{definition} Given a finite group $W_0$ we denote by $\mathbb{Z}[W_0]$ the group ring of $W_0$ over the ring of integers. 
A $\mathbb{Z}[W_0]$-lattice is a $\mathbb{Z}[W_0]$-module which is, in addition, finitely generated and $\mathbb{Z}$-torsion free. 
\end{definition}

We observe that for a semidirect product $\mathbb{Z}^n\rtimes W_0$, with $W_0:=\{w_1,\ldots, w_m\}$ finite, the subgroup $\mathbb{Z}^n$ is a $\mathbb{Z}[W_0]$-module. Indeed, we may define the scalar multiplication $r.\bar{a}=(k_1w_1+\cdots +k_mw_m).\bar{a}=k_1\bar{a}^{w_1}+\cdots+k_m\bar{a}^{w_m}$, for any $r\in \mathbb{Z}[W_0]$ and any $\bar{a}\in\mathbb{Z}^n$. Moreover, since $\mathbb{Z}^n$ is finitely generated and $\mathbb{Z}$-torsion free, it is a $\mathbb{Z}[W_0]$-lattice. 
The advantage of working with $\mathbb{Z}[W_0]$-modules is that, in some sense, they encode the action of $W_0$ on $\mathbb{Z}^n$. The proof of the following lemma is left to the reader.

\begin{lemma}\label{MultipAction}
Let $G$ be isomorphic to the semidirect product $\mathbb{Z}^n\rtimes W_0$, with $W_0$ finite and $A$ the corresponding $\mathbb{Z}[W_0]$-lattice. Let $B$ be a $\mathbb{Z}[W_0]$-module isomorphic to $A$ (as $\mathbb{Z}[W_0]$-modules). Then $B\rtimes_\alpha W_0$, with $\alpha:W_0\rightarrow \mrm{Aut}(B)$ given by $w\mapsto (b\mapsto w.b)$ for any $b\in B$ is isomorphic to $G$. In particular, if $m$ is a positive integer and $m\mathbb{Z}^n:=\{ma \ | \ a\in \mathbb{Z}^n\}$, then  $\mathbb{Z}^n\rtimes W_0\cong m\mathbb{Z}^n\rtimes W_0$, for every $m<\omega$.  
\end{lemma}

\medskip  We will work in the more general context of crystallographic space groups, which we now define. 

\begin{definition}(\cite[Definition 2.3.10]{plesken})
A space group $G$ is a group that satisfies the following conditions:
\begin{enumerate}[(1)]
\item $G$ admits a short exact sequence $A\rightarrow G\rightarrow W_0$ with $A$ Abelian and $W_0$ finite.
\item $W_0$ acts faithfully on $A$ as follows: for $w$ in $W_0$ we let some (any) pre-image of $w$ in $G$ act on (the image of) $A$ by conjugation.
\end{enumerate}
If, in addition, $A$ is free Abelian of finite rank, we call $G$ a crystallographic space group.
Moreover, the image of $A$ in $G$ is called the translation subgroup of $G$ and the image of $W_0$ in $Aut(A)$ is called the point group.    
\end{definition}

It is straightforward that irreducible affine Coxeter groups are crystallographic space groups. 
 	
\begin{definition}\label{same_genus} \
\begin{enumerate}[(i)]
    \item Let $S_1, S_2$ be crystallographic space groups and $T(S_1), T(S_2)$ their translation subgroups. Then $S_1, S_2$ belong to the same genus if $S_1/mT(S_1)\cong S_2/mT(S_2)$ for all $m<\omega$.
    \item Let $W_0$ be a finite group and $L_1, L_2$ be two $\mathbb{Z}[W_0]$-lattices. We say that $L_1$ and $L_2$ belong to the same genus (as $\mathbb{Z}[W_0]$-modules) if for every prime number $p$ and $k \in \mathbb{N}$ we have that $L_1/p^kL_1 \cong L_2/p^kL_2$ as $\mathbb{Z}[W_0]$-modules.
\end{enumerate}

\end{definition}
\begin{fact}(\cite[Exercise 1, p. 96]{plesken}\label{Exercise1}
Let $S_1, S_2$ be crystallographic space groups in the same genus. Let $T(S_1), T(S_2)$ be the corresponding  lattice groups and $W_1, W_2$ be the corresponding point groups. Then, after identifying their (isomorphic) point groups with $W$, $T(S_1)$ and $T(S_2)$ lie in the same genus as $\mathbb{Z}[W]$-lattices.     
\end{fact}
	
\begin{fact}[{\cite[Theorem~4.1.8]{plesken}}]\label{plesken_fact} Let $W_0$ be a finite group and $L_1, L_2$ be two $\mathbb{Z}[W_0]$-lattices. Suppose that $L_1, L_2$ belong to the same genus. Then for every $m \in \mathbb{N}$ there is a $\mathbb{Z}[W_0]$-monomorphism $\phi: L_2 \rightarrow L_1$ such \mbox{that $[L_1 : \phi L_2]$ has index coprime to $m$.}
\end{fact}

\begin{theorem}\label{suff_cond_rigidity} Let $W = \mathbb{Z}^n \rtimes_\alpha W_0$ with $W_0$ finite and $\alpha:W_0\rightarrow \mathbb{Z}^n$ injective. Suppose that there are subgroups $N_1, ..., N_k$ of $T = \mathbb{Z}^n \times \{e\} \leq W$ which are normal in $W$ and of finite index, $[T:N_i]=n_i$, for $1\leq i\leq k$, in T, and such that any normal subgroup of $W$ contained in $T$ is a multiple, $mN_i$, of $N_i$ for some $1\leq i\leq k$, and $m<\omega$. 
Then $W$ is first-order rigid.
\end{theorem}

\begin{proof} 
Suppose that $G$ is finitely generated and elementarily equivalent to $W$, then, by \ref{rigidity_prop}, $G \cong \mathbb{Z}^n \rtimes_\beta W_0$, with $\beta:W_0\rightarrow\mathbb{Z}^n$ injective. In particular $G$ is a crystallographic space group.  

\smallskip \noindent Let $L_1 = \mathbb{Z}^n \times \{e\} \leq W$ and $L_2 = \mathbb{Z}^n \times \{e\} \leq G$ be the corresponding translation lattices. Clearly, $L_1$ and $L_2$ are naturally $\mathbb{Z}[W_0]$-lattices. We will show that they belong to the same genus. Indeed, by Proposition \ref{the_prop_def_transla}, there exists a $\emptyset$-definable formula whose solution set in $W$ (correspondingly $G$) is the translation lattice $L_1$  (correspondingly $L_2$). Thus, the conditions $W/mL_1\cong G/mL_2$, for every $m<\omega$, are first-order expressible and since $G$ is elementarily equivalent to $W$, they belong to the same genus. Therefore, by Fact \ref{Exercise1}, $L_1, L_2$ belong to the same genus as $\mathbb{Z}[W_0]$-lattices.


\smallskip \noindent Recalling the hypotheses on the normal subgroups of $W$, let $m = \prod_{1\leq i \leq k} n_i$. Then, applying Fact~\ref{plesken_fact} for this $m$ we get a $\mathbb{Z}[W_0]$-monomorphism $\phi: L_2 \rightarrow L_1$ such that $[L_1 : \phi (L_2)]$ has index coprime to $m$. Since, $\phi$ is a $\mathbb{Z}[W_0]$ monomorphism, we get that $\phi(L_2)$ is a normal subgroup of $W$ and so, by the hypotheses, $\phi(L_2) = \ell N_j$, for some $1\leq j \leq k$ and $\ell < \omega$. We now take cases:  \\ 
\begin{itemize}
    \item[\underline{Case 1}:] If $N_j=L_1$, then $\phi(L_2)=\ell L_1$. We recall that, by Lemma \ref{MultipAction}, $W_0$ acts on $L_1$, via $\alpha$, as it acts on $\ell L_1$. In particular, $W_0$ acts on $L_2$, via $\beta$, as $W_0$ acts on $L_1$, via $\alpha$, and so $G$ is isomorphic to $W$.   
    \\
    \item[\underline{Case 2}:] If $N_j\neq L_1$, then $[L_1:N_j] = n_j>1$ and $[W:\ell N_j]=[W:N_j][N_j:\ell N_j]=n_j\ell^n$, a contradiction since it is not coprime to $m$.
\end{itemize}

\end{proof}

\begin{fact}[{\cite[Proposition~7.2]{maxwell_normal}}]\label{maxwell_normal_fact} Let $W$ be an irreducible affine Coxeter group and let $W = T \rtimes W_0$ be as in \ref{fact_affine}, then there are finitely many normal subgroups $N_1, ..., N_k$ of $W$ contained in $T$ of finite index in $T$ such that any normal subgroup of $W$ contained in $T$ is a multiple, $\ell N_i$, of $N_i$ for some $1\leq i\leq k$ and $\ell<\omega$.
\end{fact}

\begin{theorem}\label{fo_rigidity} 
Any irreducible affine Coxeter group is first-order rigid.
\end{theorem}

\begin{proof}
This follows from \ref{fact_affine}, \ref{suff_cond_rigidity} and \ref{maxwell_normal_fact}.
\end{proof}

Finally Corollary \ref{RigProf} implies that:

\begin{theorem}\label{pro_rigidity} 
Any irreducible affine Coxeter group is profinitely rigid.
\end{theorem}


\section{Hyperbolic right-angled Coxeter groups}\label{HRACG}
In this section we deal with a class of Coxeter groups almost orthogonal to affine Coxeter groups, the class of hyperbolic right-angled Coxeter groups. Particular examples in this class are f.g. universal Coxeter groups. These are groups isomorphic to finite free products of $\mathbb{Z}_2$'s. The universal Coxeter group of rank $n$, denoted by $W_n$, is the free product of $n$ copies of $\mathbb{Z}_2$. It admits the following presentation $\langle e_1, \ldots, e_n\ | \ e_1^2, \ldots, e_n^2\rangle$ and $(W_n, \{e_1, \ldots, e_n\})$ is obviously a Coxeter system.
As we already observed f.g. universal Coxeter groups are hyperbolic but even more is true, f.g. universal Coxeter groups are virtually free. One can witness this by the obvious epimorphism $\pi: W_n\rightarrow \mathbb{Z}_2$ sending each factor isomorphically to $\mathbb{Z}_2$. In particular, the subgroup of even length (reduced) words is a free subgroup of index $2$. More generally, it is known that any group that splits as a graph of groups with finite vertex groups, is virtually free, and actually this condition characterizes virtually free groups. 
Finally, the following well-known fact follows easily by the residual finiteness of $W_n$. 

\begin{fact} 
Let $n<\omega$. Then $W_n$ is Hopfian.
\end{fact}

As a matter of fact all hyperbolic groups are Hopfian. 
The following lemma follows easily if one observes that any order two element of $W_n$ can be conjugated to $e_i$ for some $1\leq i \leq n$. 

\begin{lemma}\label{Orbits} 
The orbit $Aut(W_n).e_1$ is $\emptyset$-definable by a quantifier-free formula. 
\end{lemma}
\begin{proof}
We will show that the quantifier-free formula $\phi(x):= (x^2=1 \land x\neq 1)$, defines $Aut(W_n).e_1$. Indeed, the formula $\phi(x)$ defines all order two elements in $W_n$, and if $b\in Aut(W_n).e_1$, then $b$ has order $2$. We, next, observe, by looking at the standard simplicial tree $\langle e_1\rangle*\ldots *\langle e_n\rangle$ acts on, that every element of order two can be conjugated to $e_i$, for some $i\leq n$. Hence, there is an automorphism taking this element to $e_1$ and therefore it belong to $Aut(W_n).e_1$.   
\end{proof}

Any finitely generated group admits a {\em Grushko decomposition} $A_1*A_2*\ldots *A_n*\mathbb{F}_m$, where the $A_i$'s are freely indecomposable and not infinite cyclic and $\mathbb{F}_m$ is a free group of rank $m$. Moreover, up to conjugation and permutation of the factor groups, $m$, $n$ as well as the $A_i$'s are unique.

\begin{fact}[Kurosh Subgroup Theorem]
Let $G=G_1*\ldots * G_n$ be a free product and $H$ a subgroup of $G$. Then $H$ is a free product of the form $H_1*\ldots*H_m*\mathbb{F}$, where each $H_i$ is of the form $H\cap G_{j_i}^{g_{j_i}}$, for $G_{j_i}$ some free factor of $G$ among the $G_i$, $i\leq n$  and $\mathbb{F}$ a free group.
\end{fact}

Recall that a finitely generated group is one-ended if it does not split as an HNN-extension or a non-trivial amalgamated free product with finite edge group (see \cite{Sta71}). One can relativize the notion of 
one-endeness by defining a (f.g.) group $G$ to be one-ended relative to a subgroup $H$ if $G$ does not split non-trivially over a finite subgroup where $H$ is conjugate in a vertex group of the splitting. 
Moioli (Andr\'e in the relative case) proved that one-ended (relatively one-ended) hyperbolic groups are co-Hopfian. This is the content of the next fact.

\begin{fact}[Andr\'e \cite{And22}, Moioli \cite{Moi13}]
Let $G$ be a hyperbolic group and $H$ a (maybe trivial) finitely generated subgroup of $G$. Suppose $G$ is one-ended relative to $H$. Then every monomorphism of $G$ that fixes $H$ pointwise is an automorphism.
\end{fact}

We recall that a set of subgroups $\{A_1, \ldots, A_n\}$ of a group $G$ is called {\em malnormal}, if whenever $A_i\cap A_j^g\neq \{1\}$, then $i=j$ and $g\in A_i$. Straightforwardly, the set of free factors of a free product is a malnormal set of subgroups.

\begin{theorem}\label{CoxHom2}
The universal Coxeter group of rank $2$ is $\exists$-homogeneous.
\end{theorem}
\begin{proof}
Let $\bar{a},\bar{b}$ be two tuples in $W_2$ with $tp_{\exists}(\bar{a})=tp_{\exists}(\bar{b})$. Consider $\bar{a}(x_1, x_2)$ a tuple of words in the variables $x_1, x_2$ such that $\bar{a}(e_1, e_2)=\bar{a}$ and likewise for $\bar{b}(\bar{x})$. The following existential formula will ensure that an injective endomorphism of $W_2$, taking $\bar{a}$ to $\bar{b}$, exists. 

$$\phi(\bar{u}):=\exists x_1, x_2 (\bigwedge_{i\leq 2} x_i^2=1 \land x_1\neq x_2 \bigwedge_{i\leq 2} x_i\neq 1 \land \bar{u}=\bar{a}(\bar{x})) $$ 
We obviously have that $W_2\models \phi(\bar{a})$, since we can plug in the standard generators in the existentially quantified variables and get $\bar{a}$. By the equality of $\exists$-types we have that $W_2\models \phi(\bar{b})$. Hence there exists an endomorphism $h:W_2\rightarrow W_2$  that sends $\bar{a}$ to $\bar{b}$. Moreover, since no $e_i$ is killed this endomorphism is injective on the factors of $W_2$. Hence, by the Kurosh subgroup theorem, each factor is sent to a conjugate of some factor, but, because $x_1\neq x_2$ in the above formula, no two factors are sent to the same conjugate (it could be though that the two distinct factors are sent to distinct conjugates of the same factor). Therefore $h(W_2)=\langle a_1^{g_1}, a_2^{g_2} \rangle$, where, $a_i$, for $i\leq 2$, are standard generators (not necessarily distinct). Now, the tree which $W_2$ acts on, corresponding to the natural free splitting $\mathbb{Z}_2*\mathbb{Z}_2$, is a line. The minimal subtree which $h(W_2)$ acts on is again the whole line, which is actually the axis of the hyperbolic element $a_1^{g_1}\cdot a_2^{g_2}$. Note that $a_1^{g_1}\cdot a_2^{g_2}$ is hyperbolic because by the malnormality of the free factors we have $\langle a_1^{g_1}\rangle\cap \langle a_2^{g_2}\rangle=\{1\}$. Now, if $a_1$, $a_2$ are distinct generators, then the ``fundamental domain" of the action is the (closed) segment connecting the fixed points of $a_1^{g_1}$, $a_2^{g_2}$ (including these vertices). Since all edges have trivial stabilizers we get that $h(W_2)$ is the free product $A_1^{g_1}*A_2^{g_2}$. 
 \begin{figure}[ht!]
\centering
\includegraphics[width=0.8\textwidth]{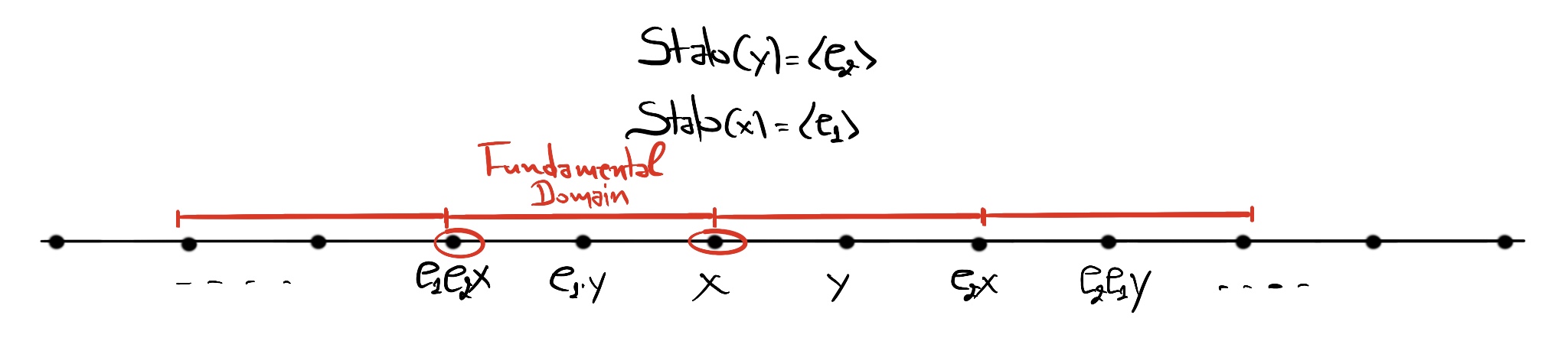}
\caption{The ``fundamental domain" of $\langle e_1, e_1^{e_1e_2}\rangle$.}
\end{figure}
We now assume that $a_1=a_2=e_1$ and without loss of generality $g_1=1$. In this case, we might have that the fundamental domain of the action includes just one of the fixed vertices and not the whole (closed) segment connecting the vertex fixed by $e_1$ and the vertex fixed by $e_1^{g_2}$. In particular, this implies that $g_2$ belongs to $\langle e_1, e_1^{g_2}\rangle$. But this is not possible since $g_2$ can be assumed to start with $e_2$ (if not $\langle e_1, e_1^{e_1g_2}\rangle$ generates the same group and $e_1g_2$ starts with $e_2$) and every reduced element of $\langle e_1, e_1^{g_2}\rangle$, involving the generator $e_1^{g_2}$, has length strictly larger than $g_2$. Hence, again $\langle e_1, e_1^{g_2}\rangle$ is the free product $\langle e_1\rangle * \langle e_1^{g_2}\rangle$.

\smallskip \noindent Consequently, in all cases, $h(W_2)\cong W_2$ and by Hopfianity of $W_2$, we get that $h$ is injective. The analogous argument shows that we have an injective endomorphism $f$ of $W_2$ that sends $\bar{b}$ to $\bar{a}$. We can now take $f\circ h$ that fixes $\bar{a}$ and is injective. We skip the proof in cases.
\\
\begin{itemize}

\item[\underline{Case 1}:] $W_2$ is one-ended with respect to $\bar{a}$ (or rather the subgroup generated by~$\bar{a}$).
\newline By (relative) co-Hopfianity $f\circ h$ is an automorphism. Hence $h$ is an automorphism and we are done. 
\\ 
\item[\underline{Case 2}:] $W_2$ is not one-ended with respect to $\bar{a}$.
\newline Suppose that $W_2$ is an amalgamated free product $A*_C B$ or an HNN-extension $A*_C$ with a finite edge group and  $\bar{a}$ belongs to a vertex group. It is fairly easy to see that if a group splits as an HNN-extension, then it admits an epimorphism onto $\mathbb{Z}$, but this is impossible for $\mathbb{Z}_2*\mathbb{Z}_2$ (and similarly for any other free product of finite groups). Hence, we can restrict our attention to amalgamated free products. Since any finite non-trivial subgroup in $W_2$ is isomorphic to $\mathbb{Z}_2$ the edge group of the splitting must be either trivial or isomorphic to $\mathbb{Z}_2$. We assume, for a contradiction, that the edge group is nontrivial. 
Kurosh Subgroup Theorem implies that each of the factors $A$ and $B$ is a free product of conjugates of free factors $\mathbb{Z}_2$ of $W_2$ together with a free group. The edge group, being finite non-trivial, must embed in both sides into a non-free factor, hence $A$ and $B$ both do not have a free group in their decomposition. Otherwise $A*_CB=(\mathbb{F}*A_1*A_2*\ldots *A_n)*_C(\mathbb{F}'*B_1*B_2*\ldots *B_m)$ admits an epimorphism to $\mathbb{F}*\mathbb{F}'$, which is impossible for a free product of finite groups. 
We next, without loss of generality, assume that $\bar{a}$ belongs to $A$ and (up to applying an inner automorphism on $B$) the edge group $C$ embeds onto some $B_i$, for $i\leq m$. In particular, the amalgamated free product may be viewed as a free product $A*B_1*B_2*\ldots B_{i-1}*B_{i+1}*B_m$. The latter is $W_2$, hence we either have $n=$1 and $m=2$ or $n=2$ and $m=1$. But in either case $\bar{a}$ belongs to a proper free factor of $W_2$. 
Hence, if $W_2$ is not one-ended with respect to $\bar{a}$, then $\bar{a}$ belongs to a proper free factor of $W_2$ and the result follows from  Lemma \ref{Orbits}.
\end{itemize}

\end{proof}

Generalizing the last part of the proof to any $W_n$ we get the following:

\begin{lemma}\label{Splittings}
Let $W_n$ be a universal Coxeter group of rank $n$. Suppose $W_n$ splits as a nontrivial amalgamated free product $A*_CB$ over a finite group $C$. 
Then  $W_n$ splits as a free product $A*B'$ where $B'$ is a free factor of a conjugate of $B$.
\end{lemma}
\begin{proof}
By Kurosh subgroup theorem $A=A_1*\ldots*A_n*\mathbb{F}$ and $B=B_1*\ldots B_m*\mathbb{F}'$, where each $A_i$ and $B_j$ are conjugates of Grushko factors of $W_n$. 
If $C$ is trivial, then we are done. If $C$ is nontrivial, then it is isomorphic to $\mathbb{Z}_2$ and must embed in both sides onto a conjugate of a non-free factor, say $gB_ig^{-1}$, for some $i\leq m$ and $g\in B$. 
%
Finally, $A*_CB=A*_C( B_1*\ldots*B_m*\mathbb{F}')$ and up to an inner automorphism of $B$ (equivalently change of the fundamental domain of the corresponding action of $A*_CB$ on the standard simplicial tree by fixing $B$ and taking $A$ to $gAg^{-1}$) we may assume that $C$ embeds onto a free factor $B_i$. In particular,  $A*_CB=(A_1*\ldots*A_n)*g^{-1}(B_1*B_2*\ldots *B_{i-1}*B_{i+1}*\ldots B_m)g$ as needed. 
\end{proof}

The following lemma is an immediate corollary of Proposition 3.23 in \cite{MPS22}.

\begin{lemma}\label{PingPong}
Let $W_n:=\langle e_1 \ | \ e_1^2\rangle*\ldots *\langle e_n \ | \ e_n^2\rangle$ be the universal Coxeter group of rank $n$. Let $A_1, A_2,\ldots, A_m$, for $m\leq n$, be conjugates of distinct factors of $W_n$. Then $\langle A_1, \ldots, A_m\rangle=A_1*A_2*\ldots *A_m$. 
\end{lemma}

\begin{theorem}
Let $W_n:=\langle e_1 \ | \ e_1^2\rangle*\ldots *\langle e_n \ | \ e_n^2\rangle$ be the universal Coxeter group of rank $n$. Then $W_n$ is homogeneous.
\end{theorem}
\begin{proof}
Let $\bar{a},\bar{b}$ be two tuples in $W_n$ with $tp_{\exists\forall}(\bar{a})=tp_{\exists\forall}(\bar{b})$. Consider $\bar{a}(x_1, x_2, \ldots, x_n)$ a tuple of words in the variables $x_1, x_2, \ldots, x_n$ such that $\bar{a}(e_1, e_2, \ldots, e_n)=\bar{a}$ and likewise $\bar{b}(\bar{x})$. The following $\exists\forall$-formula will ensure that an injective endomorphism of $W_n$ exists:
$$\phi(\bar{u}):=\exists x_1, x_2, \ldots x_n \bigl(\bigwedge_{i\leq n} x_i^2=1 \land \forall z\bigwedge_{i\neq j \atop{ i,j\leq n}} x_i^z\neq{x_j}\bigwedge_{i\leq n} x_i\neq 1 \land \bar{u}=\bar{a}(\bar{x})\bigr) $$ 

\smallskip \noindent We obviously have that $W_n\models \phi(\bar{a})$, since we can plug in the standard generators in the existentially quantified variables and get $\bar{a}$. By the equality of $\exists\forall$-types we have that $W_n\models \phi(\bar{b})$. Hence there exists an endomorphism $h$  that sends $\bar{a}$ to $\bar{b}$. Moreover, since no $e_i$ is killed this endomorphism is injective on the factors of $W_n$. Hence, by the Kurosh subgroup theorem, and the $\forall$-clause in the above formula, each factor is sent to a conjugate of a distinct factor. Therefore, by Lemma \ref{PingPong}, $h(W_n)=\langle A_1, A_2, \ldots, A_n\rangle=A_1*A_2* \ldots * A_n$. Consequently, $h(W_n)\cong W_n$ and by Hopfianity of $W_n$, we get that $h$ is injective. The analogous argument shows that we have an injective endomorphism $f$ of $W_n$ that sends $\bar{b}$ to $\bar{a}$. We can now take $f\circ h$ that fixes $\bar{a}$ and is injective. We split the proof in cases. 
\\
\begin{itemize}

\item[\underline{Case 1}:] $W_n$ is one-ended with respect to $\bar{a}$.
\newline By (relative) co-Hopfianity $f\circ h$ is an automorphism. Hence $h$ is an automorphism and we are done. 
\\
\item[\underline{Case 2}:] $W_n$ is not one-ended with respect to $\bar{a}$.
\newline
Let $W$ be the smallest free factor of $W_n$ that contains $\bar{a}$. By Lemma \ref{Splittings} $W$ must be one-ended with respect to $\bar a$. Applying the symmetric argument to $\bar{b}$, we get a splitting of $W_n$ as $W'*W'_1$ such that $W'$ is one-ended with respect to $\bar{b}$. 
We consider the restrictions $h\upharpoonright_W$ and $f\upharpoonright_{W'}$, that we still call $h$ and $f$. Since $W'$ is freely indecomposable with respect to $\bar{b}$ the image $h(W)$ is contained in $W'$ and likewise the image of $f(W')$ is contained in $W$. Hence, we get that $f\circ h$ is an injective endomorphism of $W$, therefore, by relative co-Hopfianity, an automorphism. In particular $h:W\rightarrow W'$ is an isomorphism. This isomorphism extends naturally to $W_1$ in order to obtain an automorphism of $W_n$.
\end{itemize}
\end{proof}

The same ideas work locally when the Coxeter group is hyperbolic right-angled or even globally if in addition it is one-ended. The only modification is the change of the first-order formula ensuring that the morphism sending $\bar a$ to $\bar b$ in the above proof is injective. Such formula is provided by Lemma 6.3 in \cite{MPS22}.  

\begin{proposition}
Let $G$ be a hyperbolic right-angled Coxeter group and $\bar b$ a tuple from $G$. Suppose $G$ is one-ended relative to the subgroup generated by $\bar b$. Then $\bar b$ is type-determined, i.e., whenever $tp^G(\bar b)=tp^G(\bar c)$, then there is an automorphism of $G$ mapping $\bar b$ to $\bar c$.  In particular, if $G$ is one-ended, then it is homogeneous.
\end{proposition}

\end{document}